\documentclass[11pt, oneside]{amsart}   
\usepackage{geometry}                		
\usepackage{graphicx}						
\usepackage{amssymb}
\usepackage{color}
\usepackage{enumerate}
\usepackage{MnSymbol}
\usepackage[T1]{fontenc}
\usepackage{dsfont}
\newcommand{\Ha}[2]{\medskip\noindent \textbf{(H~#1) }{\begin{em}#2\end{em}}\medskip}

\newcommand{\define}[1]{\textbf{\textit{#1}}}
\newcommand{\up}[1]{\stackrel{#1}{\longrightarrow}}
\newcommand{\wt}[1]{\langle #1 \rangle}
\newcommand{\wtilde}{\widetilde}
\newcommand{\rcl}[1]{\lsem #1\rsem}

\newcommand{\Ab}{\mathbf{Ab}}
\newcommand{\Lex}{\mathbf{Lex}}
\newcommand{\lperp}[1]{{}^\perp{#1}}
\newcommand{\incl}{\hookrightarrow}

\newcommand{\Om}{\Omega}

\renewcommand{\AA}{\mathbb{A}}
\newcommand{\CC}{\mathbb{C}}
\newcommand{\HH}{\mathbb{H}}
\newcommand{\NN}{\mathbb{N}}
\newcommand{\PP}{\mathbb{P}}

\newcommand{\RR}{\mathbb{R}}
\newcommand{\SSS}{\mathbb{S}}

\newcommand{\ZZ}{\mathbb{Z}}

\newcommand{\al}{\alpha}
\newcommand{\be}{\beta}
\newcommand{\Ga}{\Gamma}

\newcommand{\om}{\omega}
\newcommand{\si}{\sigma}
\newcommand{\Hh}{\mathcal{H}}
\newcommand{\Hhnull}{\mathcal{H}_0}
\newcommand{\Hhplus}{\mathcal{H}_+}

\newcommand{\XX}{\mathcal{X}}

\newcommand{\Tt}{\mathcal{T}}
\newcommand{\Ii}{\mathcal{I}}
\newcommand{\Oo}{\mathcal{O}}
\newcommand{\OoX}{\Oo_\XX}
\newcommand{\Ss}{\mathcal{S}}
\newcommand{\Uu}{\mathcal{U}}

\newcommand{\mM}{\frak{m}}
\newcommand{\pP}{\frak{p}}
\newcommand{\Ext}{\mathrm{Ext}}
\newcommand{\Der}{\mathrm{D}^b}
\newcommand{\coh}{\mathrm{coh}}
\newcommand{\cohnull}{\mathrm{coh}_0}
\newcommand{\vect}{\mathrm{vect}}
\newcommand{\svect}{\underline{\mathrm{vect}}}
\newcommand{\Qcoh}{\mathrm{Qcoh}}
\newcommand{\Sing}[2]{\mathrm{Sing}^{#1}(#2)}

\newcommand{\euler}[2]{\langle#1,#2\rangle}
\newcommand{\lcm}[1]{\mathrm{lcm}(#1)}
\newcommand{\mmod}{\mathrm{mod}}
\newcommand{\modgr}[2]{\mathrm{mod}^{#1}(#2)}
\newcommand{\Modgr}[2]{\mathrm{Mod}^{#1}(#2)}
\newcommand{\CMgr}[2]{\mathrm{CM}^{#1}(#2)}
\newcommand{\sCMgr}[2]{\underline{\mathrm{CM}}^{#1}(#2)}
\newcommand{\projgr}[2]{\mathrm{proj}^{#1}(#2)}
\newcommand{\modgrnull}[2]{\mathrm{mod}_0^{#1}(#2)}

\newcommand{\dg}{\mathrm{deg}}
\newcommand{\rk}{\mathrm{rk}}
\newcommand{\Aut}{\mathrm{Aut}}
\newcommand{\Hom}{\mathrm{Hom}}
\newcommand{\sHom}{\underline{\mathrm{Hom}}}
\newcommand{\End}{\mathrm{End}}
\newcommand{\Knull}{\mathrm{{K}}_0}
\newcommand{\Knullred}{\mathrm{{K}_{0}^{red}}}
\newcommand{\Pic}{\mathrm{Pic}}
\newcommand{\tsr}{\otimes}
\newcommand{\aveuler}[2]{\langle\langle #1,#2\rangle\rangle}

\theoremstyle{plain}
\newtheorem{Theorem}{Theorem}[section]
\newtheorem{Comment}[Theorem]{Comment}
\newtheorem{Definition}[Theorem]{Definition}
\newtheorem{definition}[Theorem]{Definition}
\newtheorem{Example}[Theorem]{Example}
\newtheorem{Proposition}[Theorem]{Proposition}
\newtheorem{Corollary}[Theorem]{Corollary}
\newtheorem{Remark}[Theorem]{Remark}

\title{The algebraic theory of fuchsian singularities}
\author{Helmut Lenzing}

\begin{document}
\dedicatory{Dedicated to Jos\'{e} Antonio de la Pe\~{n}a on the occasion of his 60th birthday}
\address{Universit\"{a}t Paderborn; Institut f\"{u}r Mathematik; 33095 Paderborn, Germany}
\email{helmut@math.uni-paderborn.de}

\maketitle
\sloppy
\section{Introduction}
A fuchsian singularity, as considered in this article, is classically defined --- for a discrete cocompact subgroup $G$ of the automorphism group $PSL(2,\RR)$ of the upper complex halfplane $\HH$ --- as the graded $\CC$-algebra of $G$-invariant holomorphic differential forms. These graded rings of automorphic forms relate to many mathematical objects of quite different appearance.

This article has three major aims:
\begin{enumerate}
\item Extend the notion of a fuchsian singularity to algebraically closed base field of arbitrary characteristic.
\item Discuss their relationship to mathematical objects of a different nature.
\item Provide a purely algebraic characterization of fuchsian singularities.
\end{enumerate}
In particular, the last task provides an easy, systematic, access to produce members of the rare species of fuchsian singularities that are hypersurfaces or complete intersection. We recall for this, that for underlying genus zero there are just 22 fuchsian hypersurface singularities, with 14 of them forming Arnold's famous strange duality list of exceptional unimodal singularities, compare \cite{Ebeling:2003} and \cite{Lenzing:Pena:2011}.

More specifically, and working over an algebraically closed field of arbitrary characteristic, we are going to show that there exist natural bijective correspondences between isomorphism classes of the following mathematical objects:
\begin{enumerate}[(i)]
\item Weighted smooth projective curves (WPC's) of negative orbifold Euler characteristic.
\item Hereditary noetherian categories with Serre duality (HNC's) of negative orbifold Euler characteristic.
\item $\ZZ$-graded local isolated Gorenstein singularities of Krull dimension two and Gorenstein parameter -1 (IGS's)
\item $\ZZ$-graded orbit algebras, corresponding to the $\tau$-orbit of the structure sheaf on a (WPC) of negative orbifold Euler characteristic. 
\item For $k=\CC$
\begin{enumerate}[(a)]
\item Fuchsian subgroups of  $\Aut(\HH)$, the group of biholomorphic maps of the hyperbolic plane $\HH$.
\item The graded algebras of $G$-invariant differential (automorphic) forms on the upper complex half-plane, attached to a fuchsian group $G$.
\end{enumerate}
\end{enumerate}
In Section~\ref{sect:correspondence} we provide the definitions of the above concepts and establish natural bijections between their isomorphism classes.

In Section~\ref{sect:singularities} we introduce the triangulated singularity category $\Tt$ attached to a graded fuchsian algebra $R$. Let $\XX$ be the (WPC), attached to $R$; we show that $\Tt$ is a one-point extension of the bounded derived category $\Der(\coh(\XX))$ of coherent sheaves on $\XX$. We further discuss various incarnations of the singularity category, describe their Grothendieck groups resp.\ their reduced or numerical Grothendieck groups and their relationship to the respective Grothendieck groups of $\coh(\XX)$. In particular, we establish a connection between the Coxeter polynomials for $\coh(\XX)$ and for $\Tt$ with the Poincar{\'e} series of $R$, compare \cite{Ebeling:2003} for a similar result.

Finally, in Appendix~\ref{app:1} we systematically summarize properties of coherent sheaves on smooth projective curves (resp.\ compact Riemann surfaces), and more generally on (WPC's)~(Appendix~\ref{app:2}).

This paper is a continuation of previous work with Jos\'e Antonio de la Pe\~na dealing with the special case of underlying genus zero, that is, of fuchsian singularities associated to weighted projective lines and their relationship to extended canonical algebras~\cite{Lenzing:Pena:2011}.

\section{Basic definitions and description of the correspondences}\label{sect:correspondence}

\subsection{Weighted projective curves}
\begin{definition}
A \define{weighted smooth projective curve}, weighted projective curve or (WPC) for short, is a pair $\XX=(X,w)$ consisting of a smooth projective curve $X$ and a function $w:X\to \NN_{\geq1}$ to the integers $\geq1$ such that $w(x)>1$ for only finitely many points.
\end{definition}

An isomorphism of (WPC's) $f:=(X,w)\to(X',w')$ is a biregular map $f:X\to X'$ commuting with the weight functions, that is, satisfying $w=w'\circ f$.

If $\XX=(X,w)$ is a (WPC), the points  $x_1,x_2,\ldots,x_t$ with $w(x_i)>1$ are called \define{weighted points}, the other ones \define{ordinary points} and $w_1=w(x_1),\ldots,w_t=w(x_t)$ are called the \define{weights} of $\XX$. In contexts, where the position of the weights does not matter, we also use the notation $\XX=X\wt{w_1,w_2,\ldots,w_t}$.

A (WPC) $\XX=(X,w)$ has an \define{orbifold Euler characteristic}, see Theorem~\ref{thm:OrbifoldEulerChar}, given by:
\begin{equation*}
\chi_\XX=\chi_X -\sum_{x\in X}\left(1-\frac{1}{w(x)} \right).
\end{equation*}
We note that in the above sum only the finitely many points $x$ with $w(x)>1$ matter.

\subsection{The category of coherent sheaves on a (WPC)}\label{sect:WPC}
We now describe a two step procedure associating to a (WPC) $\XX=(X,w)$ a category of coherent sheaves $\coh(\XX)$. We first form the category $\coh(X)$ of (algebraic) coherent sheaves on $X$. (If $k=\CC$ and $X$ is given as a compact Riemann surface then we form instead the category $\coh(X)$ of holomorphic coherent sheaves.) Next, for each weighted point $x_i$ of $X$ having weight $w_i$ we form, following \cite{Lenzing:Miskolc} or \cite{Lenzing:2017}, iteratively the so-called category of $w_i$-cycles. The result is an enlargement $\coh(\XX)$ of $\coh(X)$, containing $\coh(X)$ as a full subcategory, and serving as the category of coherent sheaves on $\XX$. In particular, the structure sheaf $\Oo_X$ will also be chosen as the structure sheaf $\Oo_{\XX}$ for $\XX$. Alternatively,  one can construct $\coh(\XX)$ by means of a sheaf of orders, see \cite{Reiten:VandenBergh:2002} for details.

A visible effect of the iterated cycle construction is that for an ordinary point $p$ there is still a unique simple sheaf, concentrated in $p$. By contrast, for an 'exceptional point' $x_i$ of weight $w_i$ a full Auslander-Reiten orbit of $w_i$ simple sheaves is concentrated in $x_i$. For the subcategory of vector bundles the effect of the iterated cycle construction is not so easily seen; it may however be more dramatic than for the simple sheaves. 

The category $\coh(\XX)$ then is an abelian $k$-category that is connected, noetherian $\Hom$- and $\Ext$-finite with Serre duality which satisfies all axioms of a hereditary noetherian category with Serre duality (HNC).

\subsection{Hereditary noetherian categories with Serre duality}\label{sect:HNC}
Hereditary noetherian categories with Serre duality (HNC's) have an axiomatic description. We follow here the presentation of \cite{Lenzing:Reiten:2006} which also contains the proofs of the derived statements.

By a \emph{hereditary noetherian category} we mean in the
sequel a small $k$-linear category satisfying the following assumptions:

\Ha1{$\Hh$ is a connected abelian $k$-category, and each object $A$ in
$\Hh$ is noetherian, meaning that ascending chains of subobjects for $A$ become stationary.}

\Ha2{$\Hh$ is Ext-finite, that is, all morphism and extension
spaces in $\Hh$ are finite dimensional $k$-vector spaces.}

\Ha3{$\Hh$ satisfies \define{Serre duality}
$\Ext^1(X,Y)=D\Hom(Y,\tau X)$ 
where $\tau$ is a
self-equivalence of $\Hh$, and $D$ denotes the usual $k$-dual. (The self-equivalence $\tau$ is furtheron called \define{Auslander-Reiten translation}.}

\Ha4{$\Hh$ contains an object of infinite length.}

By $\Hhnull$ we denote the Serre subcategory of $\Hh$ consisting of all objects having finite length. We then request:

\Ha5{Each object in the quotient category $\Hh/\Hhnull$ has
finite length.}

By $\Hhplus$ we denote the full subcategory of all objects of $\Hh$ not having any simple subobject. The members of $\Hhplus$ will be called \define{vector bundles}, those of $\Hh_0$ will be called \define{torsion}.
 
\begin{Proposition} \label{prop:splitting}
Each indecomposable object from $\Hh$ is either a vector bundle, that is, belongs $\Hhplus$, or it has finite length and thus is a member of $\Hhnull$. 

Moreover, for some index set $C=C(\Hh)$ the category $\Hhnull$ decomposes 
into a coproduct $\coprod_{x\in C} \Uu_x$, of connected uniserial length categories $\Uu_x$.

The simple objects in $\Uu_x$ form a single Auslander-Reiten orbit of finite cardinality $w(x)$. Only for finitely many points $x\in C(\Hh)$ the cardinality $w(x)$ is $>1$.~\hfill$\square$
\end{Proposition}

The members of $C=C(\Hh)$ are called \define{points of $\Hh$}. 

\Ha6{$\Hh$ has infinitely many points}

This completes the requests on $\Hh$. It can be shown, see ~\cite{Lenzing:Reiten:2006}, that (H 5) is implied by the other axioms.

For each object $X$ in $\Hh$ the \define{rank} $\rk(X)$ of $X$ is defined as the length of $X$ in the Serre quotient $\Hh/\Hhnull$. The rank is zero exactly for the members of $\Hhnull$. Among vector bundles, we choose a member of rank one and call it the \define{structure sheaf} $\Oo_C$ of $\Hh$. 

\begin{Proposition}
The endomorphism ring of $\Oo_C$ in the Serre quotient $\Hh/\Hhnull$ is an algebraic function field $k(\Hh)/k$ in one variable, that is, $k(\Hh)$ is a finite algebraic extension of the field $k(X)$ of rational functions in one variable.
\end{Proposition}

Classical valuation theory, see \cite{Chevalley:1951}, associates to such a function field a smooth projective curve $C$ whose points are in one-to-one correspondence with the discrete valuations of $k(\Hh)/k$. It is not difficult to show that these are in one-to-one correspondence with the members of $C(\Hh)$. In this way, $C(\Hh)$ becomes a smooth projective curve, equipped with the weight function $w$ from Proposition~\ref{prop:splitting}.

An important tool in investigating $\Hh$ is the \define{Euler form}, a bilinear form on the Grothendieck group $\Knull(\Hh)$ modulo short exact sequences. It is defined on classes of objects by the expression
\begin{equation*}
\euler{[E]}{[F]}=\dim_k\Hom(E,F)-\dim_k\Ext^1(E,F).
\end{equation*}

We are now going to describe how to get back from $\coh(\XX)$ to $\coh(X)$, where $x_1,x_2,\ldots,x_t$ are the weighted points and $w_1,w_2,\ldots, w_t$ are their weights. One uses that for each $\tau$-orbit of simple sheaves concentrated in $x_i$, $i=1,\ldots,t$, the structure sheaf $\Oo_{\XX}$ admits a non-zero homomorphism to exactly one member of each  $\tau$-orbit $(\tau^jS_{x_i})$, $j\in\ZZ_{w_i}$. 

We may assume that $\Hom(\Oo_\XX,\Ss_{x_i})=k$  for $i=1,\ldots,t$. We then form the finite set $\frak{S}$ of all simple sheaves $S$ with $\Hom(\Oo_\XX,S)=0$. 

The full subcategory of $\coh(\XX)$, right perpendicular to $\frak{S}$ is then (equivalent to) the original category $\coh(X)$ of coherent sheaves on the underlying smooth projective curve, we started with, compare~\cite{Geigle:Lenzing:1991}. By construction, the simple sheaves $S_{x_i}$ with $\Hom(\Oo_{\XX},S_{x_i})=k$, $i=1,\ldots,t$, belong to $\frak{S}^\perp=\coh(X)$. The points $x_1,\ldots,x_t$, together with their weights $w_1,\ldots,w_t$, then define a (WPC) $(X,w)$.

\subsection{Graded isolated Gorenstein singularities}
By a \define{positively $\ZZ$-graded $k$-algebra $R=\bigoplus_{n\geq 0}R_n$} we understand a commutative algebra $R$ which is affine, hence finitely generated as a $k$-algebra. We assume that the homogeneous components $R_n$ are finite dimensional over $k$ and satisfy the usual grading condition $R_n\cdot R_m\subseteq R_{n+m}$ for all $n,m\geq0$. We further request that  $R_0=k$, an assumption implying that $R$ is \define{graded local}, that is, has a unique maximal graded ideal $\mM=\bigoplus_{n>0}R_n$ and \define{residue class field} $R/\mM=k$.

A graded algebra $R$, as above, is said to be \define{graded isolated singularity} if for each homogeneous prime ideal $\pP$, different from the maximal ideal $\mM$, the $\ZZ$-graded localization $S=R_{\pP}$ has finite global dimension. We note that the homogeneous component $S_n$ consists of all fractions $a/b$ with $a,b$ homogeneous, say of degree $r$,$s$ respectively, with $b\notin\pP$, and such that $r-s=n$.

\begin{definition}
A positively $\ZZ$-graded commutative noetherian algebra $R$ is called \define{graded Gorenstein} if its injective dimension, calculated in the category  $\Modgr\ZZ{R}$ of all graded $R$-modules, is finite.
\end{definition}
This dimension then agrees with the \define{graded Krull dimension} $n$ of $R$, which may be defined as the maximal length of a proper chain of graded prime ideals
\begin{equation*}
0\subseteq\pP_{0}\subset \pP_1\subset\cdots \subset \pP_{n-1}\subset\pP_n=\mM
\end{equation*}
in $R$.

An important invariant for a $\ZZ$-graded Gorenstein algebra $R$ of Krull dimension $n$ is the \define{Gorenstein parameter}. To determine this invariant we consider a minimal graded injective resolution in $\Modgr{\ZZ}{R}$
\begin{equation*}
0\to R\to E^0\to E^1 \to \cdots \to E^{n-1}\to E^n\to 0
\end{equation*}
by $\ZZ$-graded injective modules $E^i$, $i=0,\ldots,n$. If $E^n$ is isomorphic to the $a$-fold degree shift of the graded injective hull $E(k)$ of the $R$-module $k=R/\mM$, then $a$ is called the \define{Gorenstein parameter} of $R$. To have Gorenstein parameter $a$ we thus request that the $n$-th cosyzygy of $R$ equals $E(k(a))=(E(k)(a)$.

In case of a graded complete intersection $R=k[x_1,x_2,\ldots,x_t]/(f_1,f_2\ldots,f_r)$, where $R$ is positively graded by assigning $x_1,x_2,\ldots,x_t$ positive degrees $d_1,d_2,\ldots,d_t$ and where the homogeneous polynomials $f_1,f_2\ldots,f_r$ of degrees $h_1,h_2,\ldots,h_r$ are assumed to form a homogenous regular sequence, the Gorenstein parameter $a$ of $R$ is determined by the expression 
\begin{equation*}
a=\sum_{j=1}^t d_j-\sum_{i=1}^r h_i.
\end{equation*}
In particular, a graded hypersurface $R=k[x_1,x_2,x_3]/(f)$ of Krull dimension 2, is graded Gorenstein of Gorenstein parameter $\sum_{i=1}^3\deg(x_i)-\deg(f)$.

\subsubsection{Graded Gorenstein isolated singularities of Krull dimension two and weighted projective curves}
We next describe how graded Gorenstein isolated singularities of \emph{Krull dimension two} relate to weighted projective curves.

\begin{Proposition} \label{prop:gor1}
Let $R$ be a graded Gorenstein isolated singularity of Krull dimension two and Gorenstein parameter $a$. Then $R$ is graded integral, that is, $x\cdot y=0$ for homogeneous elements implies $x=0$ or $y=0$. Moreover, sheafification of $R$, say by Serre construction, yields a hereditary noetherian $k$-category (HNC) with Serre duality 
\begin{equation}
\Hh=\frac{\modgr\ZZ{R}}{\modgrnull\ZZ{R}},
\end{equation}
where the AR-translation is induced by degree shift $X\mapsto X(a)$.\end{Proposition}

\begin{proof}
 Following the model of \cite{Geigle:Lenzing:1987}, see also \cite{Canonaco:2000,Canonaco:2003}, we construct a $\ZZ$-graded sheaf theory from $R$ on the $\ZZ$-graded (maximal) projective spectrum $X$, consisting of the non-zero homogeneous non-maximal prime ideals $\pP$ of $R$. There results a structure sheaf $\Oo$ on $X$ whose stalks are the graded localizations $R_\pP$, defined by means of homogeneous quotients $a/b$ with $a\in R$ and $b\notin \pP$. The stalks are graded regular local of Krull dimension one and hence are graded discrete valuation rings. Since all stalks have global dimension one, it follows that the abelian category $\Hh=\coh(X)$ of graded coherent sheaves has global dimension one. Moreover, since $R$ is graded noetherian, the category $\Hh$ is also noetherian. Finally, since $R$ is graded-Gorenstein, say of Gorenstein parameter $a$, the category $\Hh$ has Serre duality in the form $\Ext^1(A,B)=\Hom(B,A(a))$, where $A(a)$ denotes the degree shift of $A$ with the Gorenstein parameter $a\in\ZZ$, see for instance \cite{Orlov:2009}.

Alternatively, again following \cite{Geigle:Lenzing:1987}, $\Hh$ can be obtained by Serre construction as the abelian quotient category  of the category $\modgr\ZZ{R}$ of finitely generated graded $R$-modules modulo its Serre subcategory $\modgrnull\ZZ{R}$ of graded $R$-modules of finite length.

It is left to the reader to check that $\Hh$ fulfills all remaining requirements for a hereditary noetherian category (HNC) as introduced above. It then follows that $X=C(\Hh)$ carries the structure of a (WPC).
\end{proof}

\subsubsection{Gorenstein parameter $-1$}
Next, we are going to show how, conversely, each (HNC) $\Hh$ of \emph{negative orbifold Euler characteristic} gives rise to a $\ZZ$-graded Gorenstein isolated singularity of Krull dimension two and Gorenstein parameter -1. For this, we interpret $\Hh$ as the category of coherent sheaves $\coh(\XX)$ on a (WPC) $\XX=(X,w)$. Let $\Oo_\XX$ be the structure sheaf of $\XX$ and let $\tau:\Hh\to\Hh$ be the Auslander-Reiten translation of $\coh(\XX)$. Following the model of \cite{Lenzing:1994} we form the \define{$\ZZ$-graded orbit-algebra of $\tau$ on $\Oo_\XX$}:
\begin{equation}
R=\bigoplus_{n\geq0}\Hom(\Oo_\XX,\tau^n\Oo_\XX) \text{ with multiplication } r_n\cdot r_m=[\tau^m (r_n)    \circ r_m].
\end{equation}
From the definition of the multiplication it is easily shown that $R$ is an associative, positively $\ZZ$-graded algebra. To show that $R$ is also commutative, we need the concept of the \define{point shift automorphism} $\si_x$ with $x\in\XX$, compare~\cite[Section 10.3]{Lenzing:2007}.

For each $x\in\XX$ there is an automorphism $\si_x$ of $\coh(\XX)$ which on vector bundles $E$ is given by the universal extension $E\mapsto E(x)$
{\begin{equation} \label{eq:universal_extension}
0 \to E \up{x_E} E(x) \to  E_x \to 0,  \text{ where }E_x=\bigoplus_{j\in\ZZ_{w(x)}}\Ext^{1}(\tau{}^j S_x,E)\otimes_k
\tau{}^j S_x.
\end{equation}
Obviously, each $\si_x$, $x\in\XX$, induces the identity map on $\XX=C(\Hh)$. By definition, all the $\si_x$, $x\in\XX$, generate the Picard group $\Pic(\XX)$ which is a commutative group acting transitively on isomorphism classes of line bundles~\cite{Lenzing:2018}.

\begin{Proposition}\label{prop:Picard}
The Auslander-Reiten translation $\tau$ of $\coh(\XX)$ is a member of the \define{Picard group} $\Pic(\XX)$. Further, the orbit algebra $R$ is commutative.
\end{Proposition}
\begin{proof}
Since $\Pic(\XX)$ acts transitively on line bundles, there exists $\si\in\Pic(\XX)$ with $\tau(\OoX)=\si(\OoX)$. Since $\si^{-1}\tau$ fixes $\OoX$, it is an automorphism of $\XX$ which additionally fixes each point of $\XX$, and hence is (isomorphic to) the identity, and $\tau=\si$ follows. This proves the first assertion.

Now, $\tau$ is a finite product of shift functors $\si_x$ or their inverses $\si_x^{-1}$. Passing to the quotient category $\Hh/\Hhnull$, each shift functor $\si_x$ is isomorphic to the identity functor as follows from the functorial construction \eqref{eq:universal_extension}. This property extends to the product $\tau$. We may thus assume that $\tau=Id$ on $\Hh/\Hh_0$, and correspondingly $\tau^n\OoX=\OoX$ in $\Hh/\Hh_0$. Passing with the two products $v_n\circ \tau^m(u_m)$ and $u_m\circ\tau^n(u_n)$ from $\Hom(\OoX,\tau^{n+m}\OoX)$ to $\Hh/\Hh_0$, they simplify to the expressions $v_n\circ u_m$ and $u_m\circ v_n$ with $v_m,u_n$ lying in the endomorphism ring $K$ of $\OoX$ in the category $\Hh/\Hh_0$. As the function field of a smooth projective curve, $K$ is commutative, the identity $v_n\circ u_m=u_m\circ v_n$ in $K$ thus descends to the equality $v_n\circ \tau^m(u_m)=u_m\circ\tau^n(u_n)$ in $\Hom(\OoX,\tau^{n+m}\OoX)$, which establishes the claimed commutativity of $R$.
\end{proof}

In order to establish further properties of $R$ we prove an important property of the pair $(\OoX,\tau)$:

\begin{Proposition} \label{prop:ample}
If $\Hh=\coh(\XX)$ has negative Euler characteristic, the pair
$(\Oo_\XX,\tau)$ is \define{ample} in $\Hh$ in the sense of \cite{Artin:Zhang:1994}, meaning that the following two conditions are satisfied:
\begin{enumerate}[(i)]
\item For each $H$ in $\Hh$ there exists an epimorphism $\bigoplus_{i=1}^r\tau^{-n_i}\Oo_\XX\to H$ with exponents $n_i>0$.
\item For every epimorphism $H\to K$ in $\Hh$ the induced map $\Hom(\tau^{-N}\OoX,H) \to \Hom(\tau^{-N}\OoX,K)$ is surjective for large $N$.
\end{enumerate}
\end{Proposition}

\begin{proof}
Before entering the proof of items $(i)$ and $(ii)$, we recall some properties of $\Hh$.
\begin{itemize}
\item[(a)] Each object $H$ of $\Hh$ splits into a direct sum $H=H_+\oplus H_0$, where $H_0$ has finite length and the vector bundle $H_+$ has a finite filtration by line bundles, compare \cite{Geigle:Lenzing:1987}.
\item[(b)] For each non-zero object $A\in\coh(\XX)$ with $\XX=X\wt{a_1,a_2,\ldots,a_t}$ we obtain the slope identity $\mu(\tau A) =\mu(A)-\bar{a}\chi_\Hh$, compare \cite[Prop. 3.5]{Lenzing:Reiten:2006}\footnote{Due to a different normalization, the factor $\bar{a}$ is missing in \cite{Lenzing:Reiten:2006}}. Here, $\bar{a}=\lcm{a_1,a_2,\ldots,a_t}$.
\item[(c)] Given $H$ in $\Hh$, then $\Ext^1(\tau^{-N}\OoX,H)=0$ holds for large $N$. 
\item[(d)] Assume the sequence $0\to A \up{u} B \to C\to 0$ is exact where $A$ and $C$ satisfy property $(i)$ and, moreover, $C$ has finite length. Then also $B$ satisfies property $(i)$.
\end{itemize}
In order to prove (c) it is sufficient, in view of (a),  to assume that $H$ is either a simple sheaf or a line bundle. Invoking Serre duality, the expression $\Ext^1(\tau^{-N}\OoX,H)=D\Hom(H,\tau^{-N+1}\OoX)$ is zero if $H$ is simple (this needs no assumption on $N$). If $H$ is a line bundle, then we may choose $N$ large enough that $\dg(H)>\dg(\tau^{-N+1}\OoX)$, implying that $\Hom(H,\tau^{-N+1}\OoX)$ is zero for large $N$.

Concerning $(d)$, let $\bigoplus_{i=1}^c\tau^{-n_i}\OoX\up{v}C$ be an epimorphism. Since $\tau$ is periodic on $\Hh_0$, we may choose the $n_i$ large enough that $\Ext^1(\tau^{-n_i}\OoX,A)=D\Hom(A,\tau^{-n_i+1}\OoX)$ equals zero. Hence the map $v$ lifts to a map $\bar{v}:\bigoplus_{i=1}^c\tau^{-n_i}\OoX\to B$ and, clearly, then $(u,\bar{v}):A\oplus \bigoplus_{i=1}^c\tau^{-n_i}\OoX\up{(u,\bar{v})} C$ is an epimorphism. We are done, since $A$ satisfies $(i)$.

We next prove property $(ii)$: The epimorphism $u:H\to K$ yields an exact sequence $0\to N\to H\up{u} K\to 0$. Invoking $(c)$, $\Ext^1(\tau^{-N} \OoX,N)=0$ for large $N$, implying that 
$\Hom(\tau^{-N}\OoX,u)$ is surjective.

Finally we prove property $(i)$, where we first deal with the cases where $H$ either has finite length or where $H$ is a vector bundle. 
Since $\OoX$ maps nontrivially to a member $S$ of each $\tau$-orbit of simples, it is obvious that there exists an epimorphism $\tau^{-n}\OoX\to S$ with $n\geq0$ which deals with the case of a simple sheaf. If $H$ has finite length, we consider a sequence $0\to S \to H \to C \to 0$, where $S$ is simple, and, using induction on the length, apply $(d)$.

If $H=L$ is a line bundle, we invoke negative Euler characteristic and the weighted Riemann-Roch theorem~\ref{thm:WeightedRR} to establish a non-zero morphism $u:\tau^{-n}\OoX\to L$ for some $n>0$. In the resulting exact sequence $0\to \tau^{-n}\OoX\up{u} L\to C\to 0$, the cokernel-term has finite length. Invoking the preceding argument, property $(d)$ shows that $L$ satisfies property $(i)$. 

In view of property $(a)$ this finishes the proof.
\end{proof}

\begin{Theorem}\label{thm:IGS}
Let $\XX=(X,w)$ be a smooth weighted projective curve (WPC) of negative orbifold Euler characteristic $\chi_\XX$ with category $\Hh=\coh(\XX)$ of coherent sheaves and structure sheaf $\OoX$. With respect to the AR-translation of $\Hh$ let
\begin{equation*}
R=\bigoplus_{n\geq0}\Hom(\OoX,\tau^n\OoX)
\end{equation*}
be the positively $\ZZ$-graded orbit algebra of $\tau$ on $\OoX$.

Then $R$ is a commutative affine $k$-algebra which is a graded isolated singularity of Krull dimension two. Morover $R$ is $\ZZ$-graded Gorenstein of Gorenstein parameter $-1$.

Sheafification of $R$ by means of Serre construction or by a $\ZZ$-graded sheaf theory leads back to $\coh(\XX)$. Moreover, the degree shift by $-1$ on $\modgr\ZZ{R}$ induces on $\coh(\XX)$ the AR-translation $\tau$.
\end{Theorem}

\begin{proof}
In Proposition~\ref{prop:Picard} we have shown that $R$ is commutative. Since the pair $(\OoX,\tau)$ is ample in $\coh(\XX)$ it follows from \cite[Theorem 4.5]{Artin:Zhang:1994} that $R$ is graded noetherian and sheafification of $R$ leads back to $\coh(\XX)$. Since $R_0=k$, $R$ is graded local and a finite system $x_1,x_2,\ldots,x_s$ of generators for the graded maximal ideal $\mM$ of $R$ generates $R$ as a $k$-algebra, which is thus affine over $k$. Serre construction shows that the category $\modgr\ZZ{R}$, hence $R$, has graded Krull dimension two. Using a $\ZZ$-graded sheaf theory --- instead of Serre construction --- shows that the graded localizations $R_\pP$ for non-maximal primes, that is, just the stalks of $\OoX$ are hereditary, which shows that $R$ is graded isolated singularity.

Before showing the claimed Gorenstein properties, we recall the concept of a graded dual: If $M:=\bigoplus_{n\in\ZZ}M_n$ is a graded module, then $DM=\bigoplus_{n\in\ZZ}DM_{-n}$ is called the \define{graded dual} of $M$. Since $R$ is indecomposable graded projective it follows that $DR=\bigoplus_{n\leq0}DR_{-n}$ is indecomposable graded injective, and since $DR_0=k$ is contained in $DR$, it is the graded injective envelope  $E(k)$ of $k$. We next show that $R$ is graded Gorenstein of Gorenstein parameter $-1$. For this we start
with a minimal injective resolution
\begin{equation} \label{eq:injres}
0\to \OoX \to \Ii^0 \to \Ii^1\to 0
\end{equation}
in the Grothendieck category $\Qcoh(\XX)$ of quasi-coherent sheaves. Such a resolution exists since $\Qcoh(\XX)$ has global dimension one because $\coh(\XX)$ has the same property and is additionally noetherian\footnote{Even, if one only knows the noetherian category $\Hh$ of coherent sheaves, the category of quasicoherent sheaves is easily constructed as the Grothendieck category $\Lex(\Hh^{op},\Ab)$ of left exact additive functors on $\Hh^{op}$ with values in the category $\Ab$ of abelian groups, see \cite{Gabriel:1962}.}.
We apply graded global sections $\Ga^*=\bigoplus_{n\in\ZZ}\Hom(\tau^{-n}\OoX,-)$ to \eqref{eq:injres} and obtain exactness of 
\begin{equation}
0\to R \to \Ga^*(\Ii_0)\to \Ga^*(\Ii^1)\to \bigoplus_{n\in\ZZ} \Ext^1(\tau^{-n}\OoX,\OoX)\to 0
\end{equation}
in the category $\Modgr\ZZ{R}$  of all $\ZZ$-graded $R$-modules. Due to Serre duality the Ext-term evaluates to $\bigoplus_{n\in\ZZ} D\Hom(\OoX,\tau^{-(n-1)}\OoX)$, the graded dual $(DR)(-1)$ of $R(-1)$, thus the graded injective envelope of $k(-1)$. This establishs that $R$ has graded injective dimension two with Gorenstein parameter $-1$.
\end{proof}

\begin{Corollary}[Uniqueness]
Assuming negative orbifold Euler characteristic of $\XX$, a
positively $\ZZ$-graded isolated Gorenstein singularity $S$ of Krull dimension two and Gorenstein parameter $-1$ is isomorphic to the orbit algebra $R=\bigoplus_{n\geq0}\Hom(\Oo_\XX,\tau^n\Oo_\XX)$ if and only if its associated category of coherent sheaves is equivalent to $\coh(\XX)$.
\end{Corollary}
\begin{proof}
Assume Serre construction for $S$ yields the category $\coh(\XX)$. Then sheafification $\tilde{S}$ of $S$ yields a line bundle in $\coh(\XX)$. Applying a suitable automorphism of $\coh(\XX)$, we may assume that $\tilde{S}=\OoX$. Since $S$ is graded Gorenstein of Gorenstein parameter $-1$, the degree shift by $-1$ is equal to $\tau$, and then 
\begin{equation}
S_r=\Hom(\tilde{S},\tilde{S}(r)=\Hom(\OoX,\tau^r\OoX)=R_r
\end{equation}
yielding an isomorphism between $R$ and $S$ as graded algebras.
\end{proof}
For $\XX$ a (WPC) of negative Euler characteristic we call this uniqely determined graded isolated singularity of Gorenstein parameter $-1$ the \define{fuchsian singularity associated to $\XX$ or $\coh(\XX)$}.

If $R$ is a fuchsian singularity with attached weighted projective curve $\XX=(X,w)$ the \define{signature} of $R$ is defined as the tuple $(g_X;a_1,a_2,\ldots,a_t)$, where $g_X$ is the genus of $X$ and $(a_1,a_2,\ldots,a_t)$ is the weight sequence of $\XX$. We are next showing a couple of polynomials $f$ in three variables defining \define{fuchsian  hypersurfaces} $R=k[x_1,x_2,x_3]/(f)$ and display their signatures. It was already mentioned that fuchsian hypersurfaces are a rare species. Indeed for genus $0$ there only exists finitely many and it was moreover shown by Wagreich \cite{Wagreich:1980}, see also \cite{Dolgachev:1975}, \cite{Sherbak:1978} and \cite{Lenzing:Pena:2011}, that also for higher genus there are only finitely many signatures which can be represented by a fuchsian hypersurface. Ebeling~\cite{Ebeling:2003} extended this also to complete intersections.
$$
\begin{array}{|c|c|c|c|}\hline
 &f             &deg(x,y,z|f)&signature\\\hline
1&x^2+y^3+z^7   &(21,14,6|42)   &(0;2,3,7)\\
2&x^2+y^3+yz^5  &(15,10,4|30)   &(0;2,4,5)\\
3&x^2+zy^3+yz^4 &(11,6,4|22)    &(0;2,4,6)\\
4&x^4+xy^2+yz^2 &(4,6,5|16)     &(0;2,5,6)\\
5&x^2+y^6+z^6     &(3,1,1|6)      &(2;-)   \\
6&x^4+y^4+z^4   &(1,1,1|4)      &(3;-) \\
7&xy^3+yz^3+zx^3&(1,1,1|4)      &(3,-) \\\hline
\end{array}
$$
Fuchsian singularities $R$ of genus zero with a signature having at most three weights are uniquely determined by their signature, because then the weight sequence determines the weighted projective line attached to $R$.
 
In general, however, fuchsian singularities with the same signature may not be isomorphic as show the table etems 6 and 7, the \define{Fermat quartic} and the\define{Klein's quartic}. Assuming $k=\CC$, the two curves have non-isomorphic automorphism groups: in case of the Fermat quartic this is the semi-direct product of the symmetric group $S_3$ with the action on the product of three cyclic groups of order 4 modulo their diagonal \cite{Tzermias:1995}. By contrast, the automorphism group of Klein's quartic is the nonabelian simple group $G_{168}$ of order 168, compare the papers in \cite{Levy:1999}.

\subsection{Fuchsian groups and their automorphic forms}
For this subsection, the base field is the field $\CC$ of complex numbers. Instead of weighted projective curves, we consider weighted compact Riemann surfaces (WCRS's) equipped with holomorphic coherent sheaves. For Fuchsian groups, in general, we refer to \cite{Beardon:1995} and \cite{Katok:1992}.

Within this article, a discrete subgroup $G$ of $\Aut(\HH)$, the group of holomorphic automorphisms of the upper complex halfplane $\HH$ is called \define{a fuchsian group} if it is finitely generated and cocompact, meaning that the quotient $\Hh/G$ is compact. We assume moreover that the stabilizer groups $G_x$, with $x\in\HH$, are all finite and then, automatically, cyclic, see \cite{Katok:1992}. In this case $G$ has a fundamental domain that is a hyperbolic convex polygon with finitely many sides in the hyperbollic plane $\HH$.

\begin{Proposition}[\cite{Collins:Zieschang:1998},Theorem 3.2.10]
Two fuchsian subgroups $G$ and $G'$ of $\Aut(\HH)$ are isomorphic if and only if they are conjugate as subgroups of $\Aut(\HH)$. Accordingly, the (WCRS's) $\HH/G$ and $\HH/G'$ are isomorphic if and only if $G$ and $G'$ are isomorphic.~\hfill$\square$
\end{Proposition}

For any fuchsian group $G$ the quotient $\HH/G$ is a compact holomorphic orbifold $\XX$ of complex dimension one. In other words, $\XX$ has an underlying compact Riemann surface with a finite number of cone points. These cone points, and their orders, correspond to the $G$-orbits with non-trivial stabilizer group, which are necessarily cyclic, see~\cite{Katok:1992}.

In other words, the quotient $\HH/G$ is a weighted compact Riemann surface, or just a weighted projective curve $(X,w)$ over $\CC$.

Conversely, we can pass back from a weighted Riemann surface $\XX=(X,w)$ of negative orbifold Euler characteristic to the action of a fuchsian group on $\Hh$:

We refer to \cite[Chapter 13]{Thurston:2002} and \cite{Lenzing:2017} for the concepts \define{orbifold fundamental group} and \define{universal orbifold cover}. The next theorem combines properties from \cite[Proposition~6 and Theorem~7]{Lenzing:2017}:

\begin{Theorem}\label{thm:FuchsianDeckTransformations}
Let $\XX=X\wt{a_1,a_2,\ldots,a_t}$ be a compact weighted Riemann surface (WCRS) of negative Euler characteristic $\chi_\XX$, then the following holds:
\begin{enumerate}[(i)]
\item The orbifold fundamental group $\pi_1^{orb}(\XX)$ has generators $\al_1,\al_2,\ldots,\al_g$, $\be_1,\be_2,\ldots,\be_g$, $\si_1,\si_2,\ldots, \si_t$ and is subject to the relations
$$
\si_1^{a_1}=\si_2^{a_2}=\cdots =\si_t^{a_t}=1=\si_1\si_2\cdots\si_t\,[\al_1,\be_1][\al_2,\be_2]\cdots[\al_g,\be_g],
$$
where $[a,b]$ denotes the commutator $aba^{-1}b^{-1}$ of $a$ and $b$.
\item The orbifold fundamental group $\pi_1^{orb}(\XX)$ acts on the universal orbifold cover $\HH=\wtilde{\XX}$ as group of deck transformations (the members of $\Aut\HH$ commuting with the projection $\pi:\HH=\wtilde\XX\to \XX$).  This action is discrete on $\HH$ and represents $\XX$ as orbifold quotient
  $$\XX={\wtilde{\XX}}/{\pi_1^{orb}(\XX)}.
  $$
  \item In particular, the orbifold fundamental group is fuchsian, and each fuchsian group (in the sense of the paper) arises that way.
\end{enumerate}
\end{Theorem}
\begin{proof}
For the first two items we refer to \cite[Proposition~6 and Theorem~7]{Lenzing:2017}. The last assertion follows from the other two.
\end{proof}

In the description of automorphic forms we follow Milnor~\cite{Milnor:1975a}. Let $G$ be a fuchsian group and $U$ a $G$-stable open subset of the upper complex half-plane $\HH$. A \define{differential form of degree $n$ on $U$} has the form $\Phi=f(z)z^k$. For any $g\in G$ the pull-back of $\Phi$ along $g$ is given as $g^*(\Phi)=f(g(z))\stackrel{\bullet}{g}(z)^k dz^k$, where $\stackrel{\bullet}{g}(z)$ denotes the derivative $dg(z)/dz$. Differential forms $\Phi_1$ and $\Phi_2$ of degrees $m$ and $n$ may be multiplied; this yields a differential form of degree $m+n$. The $G$-invariant differential forms of degree $n\geq0$, or \define{automorphic forms} on $G$-invariant open subsets of $\HH$, descend to the quotient $\XX=\HH/G$, yielding there holomorphic coherent sheaves of differentials $\Om^n$ for each degree $n$. 
\begin{Remark}
One can show that $\Om^n=\tau^n\OoX$, for each $n\geq0$, implying that the graded algebra
$$
\bigoplus_{n\geq0}\Ga(\XX,\Om^n)
$$
of differential forms on $\XX$ is isomorphic to the orbit algebra of the Auslander-Reiten translation of $\coh(\XX)$ on $\OoX$, thus providing a link to the classical theory of fuchsian singularities.
\end{Remark}

\section{The singularity category and its Grothendieck group(s)} \label{sect:singularities}
There are several incarnations of the singularity category of a fuchsian singularity $R$, where each one has its own advantages.

Following Buchweitz~\cite{Buchweitz:1986}, for the ungraded and Orlov~\cite{Orlov:2009} for the graded case, the \define{category of graded singularities} $\Sing\ZZ{R}$ is defined as the Verdier quotient 
\begin{equation}
\Sing\ZZ{R}=\frac{\Der(\modgr\ZZ{R})}{\Der(\projgr\ZZ{R})},
\end{equation}
where $\projgr\ZZ{R}$ denotes the category of finitely generated projetive $R$-modules.

In the case of a graded Gorenstein algebra, the singularity category $\Sing\ZZ{R}$ is equivalent to the \define{stable category of $\ZZ$-graded Cohen-Macaulay modules over $R$} $\CMgr\ZZ{R}$. In our case of graded Krull-dimension two, a finitely generated $\ZZ$-graded $R$-module $M$ is called (maximal) \define{Cohen-Macaulay} if $M$ satisfies the two conditions $\Hom(k(n),M)=0=\Ext^1(k(n),M)$ for each integer $n$. 

The category $\CMgr\ZZ{R}$ carries an exact structure in the sense of Quillen consisting of all short exact sequences with members in $\CMgr\ZZ{R}$ that are exact in $\modgr\ZZ{R}$. The full subcategory $\projgr\ZZ{R}$ of finitely generated projective modules consists exactly of the relative projectives in $\CMgr\ZZ{R}$; moreover, since $R$ is graded Gorenstein, it is known that the relative projectives coincide with the relative injectives. Further, $\CMgr\ZZ{R}$ has sufficiently many relative projectives resp.\ injectives, and thus is a \define{Frobenius category}. 

Stabilizing $\CMgr\ZZ{R}$, that is forming the factor 
category $\sCMgr\ZZ{R}=\CMgr\ZZ{R}/[\projgr\ZZ{R}]$ modulo 
all moprhisms that factor through projectives, yields a category that is triangulated. It is due to Buchweitz~\cite{Buchweitz:1986} that the functor
\begin{equation*}
\sCMgr\ZZ{R}\to \Sing\ZZ{R},\quad M\mapsto (M),
\end{equation*}
sending a graded CM-module $M$ to the stalk complex $(M)$, concentrated in degree zero, induces an equivalence of $\sCMgr\ZZ{R}$ to the singularity category $\Sing\ZZ{R}$.

In the special case of the orbit algebra $R=\bigoplus_{n\geq0}\Hom(\OoX,\tau^n\OoX)$ associated to a (WPC) of negative orbifold Euler characteristic, the category $\CMgr\ZZ{R}$ has an even more accessible incarnation as the \define{stable category of vector bundles} on $\XX$.

\begin{Proposition}
There is an equivalence between the category $\CMgr\ZZ{R}$ of finitely generated graded (maximal) Cohen-Macaulay modules over $R$ and the category $\vect(\XX)$ of vector bundles on $\XX$ induced by restricting the quotient functor $q:\modgr\ZZ{R}\to \coh(\XX)$ to the subcategory $\CMgr\ZZ{R}$. The inverse functor is given by taking graded global sections
$\Ga^*=\bigoplus_{n\in\ZZ}\Hom(\tau^{-n}\OoX,-)$ of vector bundles.
Moreover the equivalence $q:\CMgr\ZZ{R}\to\vect(\XX)$ induces an equivalence between indecomposable projective $R$-modules and members of the orbit $\tau^\ZZ\OoX$.
\end{Proposition}
\begin{proof}
The proof of the corresponding result in \cite{Geigle:Lenzing:1987} carries over.
\end{proof}
We define the \define{stable category of vector bundles} on $\XX$ as the factor category $\svect(\XX)=\vect(\XX)/[\tau^\ZZ\OoX]$ of $\vect(\XX)$ by the ideal of all morphisms factoring through a member from the additive closure of $\tau^\ZZ\OoX$.
\begin{Corollary}
The stable category $\svect(\XX)=\vect(\XX)/[\tau^\ZZ\OoX]$ is triangle equivalent to the stable category $\sCMgr\ZZ{R}$, and hence triangle-equivalent to the singulrity category $\Sing\ZZ{R}$~\hfill$\square$
\end{Corollary}

We recall that a triangulated category $\Tt$ is \define{homologically finite} if for all objects $X,Y$ from $\Tt$ we have $\Hom(X, Y [n]) = 0$ for $|n|\gg 0$.

\begin{Theorem}[Serre duality]
The singularity category $\svect(\XX)$ is triangulated, Hom-finite, Krull-Schmidt
and homologically finite. Moreover, $\svect(\XX)$ has Serre duality given by functorial isomorphisms
$\Hom(X, Y [1]) = D\Hom(Y, \tau X)$, where $\tau$ is an equivalence induced from the AR-translation of $\vect(\XX)$.
In particular, $\svect(\XX)$ has AR-triangles, and the AR-translation for $\vect(\XX)$ induces the
AR-translation for $\svect(\XX)$.
\end{Theorem}

\begin{proof} As a factor category of a Hom-finite category, $\svect(\XX)$ inherits Hom-finiteness, and the
Krull-Schmidt property from $\vect(\XX)$. Concerning Serre duality we apply
the functors $\Hom(X,-)$ (resp. $\Hom(-,\tau X)$ to the exact sequence $\mu: 0\to Y \to \Ii(Y ) \to Y [1] \to 0$,  where $\Ii(Y)$ denotes an injective hull of $Y$ in the Frobenius category $\vect(\XX)$, and obtain exact sequences\begin{align}\label{eq:seq1}
\Hom(X, \Ii(Y )) \to &\Hom(X, Y [1]) \to \sHom(X, Y [1]) \to 0,\\ 
\Hom(\Ii(Y ), \tau X) \to &\Hom(Y, \tau X) \to \sHom(Y, \tau X) \to 0. \label{eq:seq2}
\end{align}
By Serre duality in $\coh(\XX)$, dualization of the lower sequence yields exactness of
\begin{equation*}
0 \to D\sHom(Y, \tau X) \to \Ext^1(X, Y ) \to \Ext^1(X, \Ii(Y )),
\end{equation*}
where $\sHom$ denotes morphisms in $\svect(\XX)$.
Invoking the long exact Hom-Ext sequence $\Hom(X, \mu)$ we obtain from the two sequences \eqref{eq:seq1} and \eqref{eq:seq2} a natural isomorphism $\sHom(X, Y [1]) = D\sHom(Y, \tau X)$.

To show homological finiteness of $\vect(\XX)$ we use that in the Frobenius category of vector bundles the suspension functor is given by a distinguished exact sequence $0\to E \to \Ii(E) \to E[1]\to 0$, where $\Ii(E)$ denotes the injective hull in the Frobenius category. 
Because $\chi_\XX<0$, iteration of the construction leads to $E[n]$ which obtains an arbitrary large slope for $n\gg0$. Invoking a line a bundle filtration for $F$, this implies that $\Hom(E[n],F)=0$, and then also $\sHom(E[n],F)=0$ for $n>>0$. Concerning $\sHom(F,E[n])$, we invoke Serre duality for $\svect(\XX)$ and obtain $D\sHom(E[n-1],\tau F)$ which for $n\gg0$ is also zero.
\end{proof}

\begin{Corollary}
All AR-components of the singularity category have shape $\ZZ\AA_\infty$.
\end{Corollary}
\begin{proof}
Arguing as in \cite[Proposition 4.5]{Lenzing:Pena:1997} one shows first, invoking negative orbifold Euler characteristic, that all AR-components in $\vect(\XX)$ have shape $\ZZ\AA_\infty$. Next, one observes that line bundles always sit at the border-orbit of their component, which implies that passage to the stable category $\svect(\XX)$ consists in deleting the border, and thus also preserves the $\ZZ\AA_\infty$-shape of the component containing $\tau^\ZZ\OoX$.
\end{proof}

As a fuchsian singularity $R$ has Gorenstein parameter $-1$, further Serre construction $\modgr\ZZ{R}/\modgrnull\ZZ{R}$ yields the category $\coh(\XX)$ of coherent sheaves on a weighted projective curve $\XX$. It then follows from \cite{Orlov:2009} that the object $k=R/\mM$ is exceptional in $\Sing\ZZ{R}$ and its right perpendicular category, consisting of all objects $X$ in $\Sing\ZZ{R}$ satisfying $\Hom(k,X[n])=0$ for each integer $n$, is equivalent to $\Der(\coh(\XX))$. More is true, $\Sing\ZZ{R}$ behaves like a \define{one-point extension} of $\Der(\coh(\XX))$ by the structure sheaf $\OoX$:

\begin{Theorem}[One-point extension]\label{thm:onept}
Let $R$ be a fuchsian singularity with associated weighted projective curve $\XX$. Then the right perpendicular category in $\Tt=\Sing\ZZ{R}$ to the exceptional object $k=R/\mM$ is equivalent to $\Der(\coh(\XX))$.

Moreover, the right adjoint $r:\Tt\to k^\perp$ to inclusion $k^\perp\incl \Tt$ sends $k$ to $\OoX$ and, hence induces functorial isomorphisms $\Hom(X,k)=\Hom(X,\OoX)$ for each $X\in\coh(\XX)$.
\end{Theorem}

\begin{proof}
The first assertion is an immediate consequence of Orlov's theorem~\cite{Orlov:2009} since $R$ is graded Gorenstein of Gorenstein parameter $-1$.

We now turn to the second assertion. Since the Picard group acts transitively on the line bundles of $\coh(\XX)$, it suffices to show that the left
adjoint $\ell:\Tt\to \lperp{k}$ to the inclusion $j:\lperp{k}\incl
\Tt$ maps $k$ to a line bundle in $\Der(\coh(\XX))$ up to
suspension. For the proof of this claim, we refer to the related proof in \cite[Section 3.8]{Lenzing:Pena:2011}.
\end{proof}

An immediate consequence is the structure of the Grothendieck group $\Knull(\Sing\ZZ{R}$ as well as the structure of the the reduced Grothendieck group.

\begin{Proposition}\label{prop:Knull}
The Grothendieck group $\Knull(\Sing\ZZ{R})$ is the K-theoretic one-point extension of $\Knull(\XX)=\Knull(\Der(\coh(\XX))$ by the class $[\OoX]$ of the structure sheaf of $\XX$.
\end{Proposition}
\begin{proof}
In more detail this means that for the class $e=[E]$ of some exceptional object $E=k$ in $\Tt=\Sing\ZZ{R}$, we get that $e^\perp=\{x\in\Knull(\Tt)\mid \euler{e}{x}_\Tt=0\}$ equals $\Knull(\XX)$. Further $\Knull(\Tt)=\Knull(\XX)\oplus \ZZ.e$  where the Euler form $\euler{-}{-}_\Tt$ satisfies the following conditions
\begin{align*}
\euler{e}{e}_\Tt&=1\\
\euler{x}{y}_\Tt&=\euler{x}{y}_\XX \text{ if $x$ and $y$ belong to } \Knull(\XX)\\
\euler{e}{x}_\Tt&=0  \text{ and }\euler{x}{e}_\Tt=\euler{x}{[\OoX]}_\XX \text{ if $x$ belongs to $\Knull(\XX)$}\end{align*}
These properties follow immediately from Theorem~\ref{thm:onept}. \end{proof}
Let $\XX=(X,w)$ with a smooth projective curve $X$ of genus $g$. For $g\geq1$ the Grothendieck group $\Knull(\XX)$, and hence $\Sing\ZZ{R}$ are not finitely generated, compare \cite{Lenzing:2017,Lenzing:2018}. In particular, for $g\geq1$ neither $\coh(\XX)$ nor $\Sing\ZZ{R}$ has a tilting object. By contrast, for $g=0$, where we deal with weighted projective lines both categories have tilting objects, where the endomorphism rings are \define{canonical algebras} in case $\coh(\XX)$, see \cite{Geigle:Lenzing:1987}, respectively \define{extended canonical algebras} in case $\Sing\ZZ{R}$, see \cite{Lenzing:Pena:2011}. In both cases the Euler forms are non-degenerate.

For genus $g>0$, the Euler form for $\coh(\XX))$ always has a large kernel
\begin{equation*}
K=\{x\in \Knull\mid \euler{x}{y}_\XX=0 \text{ for each $x\in\Knull(\XX)$}\}.
\end{equation*}
Due to Serre duality, the kernel equivalently can be described as the set of all $y\in\Knull(\XX)$ such that $\euler{x}{y}_\XX=0$ for all $X$. Similar assertions hold for the kernel of the Euler form on $\Knull(\Sing\ZZ{R})$. It is further easily seen that $K$ is generated by all differences $[S]-[T]$, where $S,T$ run through all ordinary simple sheaves on $\XX$.

If one is mainly interested in the Euler form and related constructions, then both for the category of coherent sheaves and for the singularity category it is advisable to pass to the quotients \begin{equation*}
 \Knullred(\XX)=\Knull(\XX)/K \text{ and } \Knullred(\Sing\ZZ{R})=\Knull(\Sing\ZZ{R})/K.
 \end{equation*}
Both groups, called \define{reduced} or \define{numerical} Grothendieck groups, are finitely generated free abelian groups. Moreover, the Euler form in question $\euler{-}{-}_\XX$, respectively $\euler{-}{-}_\Tt$, descends to non-degenerate bilinear forms on the reduced groups.

\begin{Example}\label{ex:hypersurface}
To illustrate the concept, we consider a smooth projective curve (resp.\ a compact Riemann surface) $X$ of genus $g$. Then group $\Knullred(\XX)$ has a  basis $a=[\OoX]$, $b=[S]$, where $S$ is any simple sheaf on $X$. For $g\geq2$, and $R$ the orbit algebra of the AR-translation $\tau$ of $\coh(\XX)$, we have $\Knullred(\Sing\ZZ{R})=\ZZ^3$ with basis $a=[\OoX]$, $b=[S]$ with $S$ a simple sheaf, and $e=[E]$ the class of the extending object. 

With regard to the above bases, the two Euler forms $\euler{-}{-}_\XX$ and $\euler{-}{-}_\Tt$ can be expressed by their \define{Cartan matrices}
\begin{equation}
C_X=\left[ \begin{array}{cc}
1-g & -1 \\
-1 &0 
 \end{array}\right]
\text{ and }
C_\Tt=\left[ \begin{array}{ccc}
1-g & 1& 1-g\\
-1  & 0& -1 \\
0   & 0& 1
\end{array}
       \right].
\end{equation}
Further we can form the \define{Coxeter transformations} 
\begin{equation}
\Phi_X=-C_X^{-1}C_X^{tr}=\left[
\begin{array}{cc}
1 & 0  \\
2(g-1) & 1
\end{array} 
\right] 
\text{ and }
\Phi_\Tt=-C_\Tt^{-1}C_\Tt^{tr}=\left[
\begin{array}{ccc}
2-g & -1 & 1\\
2(g-1) & 1 & 0 \\
g-1 & 1 & -1
\end{array}
\right].
\end{equation}
The characteristic polynomial of the Coxeter transformation $\Phi_X$ (resp.\  $\Phi_\Tt$)  is called \define{Coxeter polynomial} of $X$ (resp.\ $\Tt$). If $X$ is a smooth projective curve of genus $g$, then the respective Coxeter polynomials are
\begin{equation}
\varphi_X=(x-1)^2 \text{ and } \varphi_\Tt=x^3+(g-2)x^2+(g-2)x+1.
\end{equation}
Note that $\varphi_X$ does not depend on $g$, while $\varphi_\Tt$ does. For $g\leq 5$ $\varphi_\Tt$ factors into cyclotomic polynomials 
\begin{equation}
g=2 \text{ then } \varphi_\Tt=\Phi_2\Phi_6,\; g=3 \text{ then } \varphi_\Tt=\Phi_2\Phi_4,\; g=4 \text{ then } \Phi_\Tt\Phi_2\Phi_3,\; g=5 \text{ then } \varphi_\Tt=\Phi_2^3.
\end{equation}
Moreover, the Coxeter transformations $\Phi_\Tt$ for $g=2$, $g=3$, $g=4$ are periodic of order $6, 4$, and $6$, respectively, while for $g=5$ it is not periodic. Since for $g\geq6$ the factorization of $\varphi_\Tt$ in irreducible factors (over $\ZZ$) is $\varphi_\Tt=\varphi_2(x^2+(g-3)x+1)$ is not a product of cyclotomic polynomials, $\Phi_\Tt$ can also not be periodic for $g\geq6$.
\end{Example}
We now comment on the weighted situation. Let $\XX=X\wt{a_1,a_2,\ldots,a_t}$ be a weighted projective curve, then the Grothendieck group $\Knull(\XX)=\Knull(\coh(\XX))$ is obtained from $\Knull(X)$ by attaching for each weight a K-theoretic tube $T_i$ of rank $a_i$, see \cite{Lenzing:2018} for details. The same procedure works to obtain the reduced Grothendieck group $\Knullred(\XX)$ from $\Knullred(X)$. It follows that the Coxeter polynomial of $\XX$ equals $\varphi_\XX=(x-1)^2\prod_{i=1}^tv_{a_i}$ which agrees with the Coxeter polynomial of the weighted projective line $\PP^1\wt{a_1,a_2,\ldots,a_t}$.

Concerning the Grothendieck group $\Knull(\Tt)$ (resp.\ the reduced Grothendieck group $\Knullred(\Tt)$) of $\Tt=\Sing\ZZ{R}$, where $R$ is the orbit algebra attached to $\XX$, the situation is more interesting. In both cases, it is obtained as the one-point extension of $\Knull(\XX)$, resp. $\Knullred(\XX)$, by the respective class of the structure sheaf $\OoX$, yielding, for instance, $\Knullred(\Tt)=\Knullred(\XX)\oplus \ZZ e$, where $\euler{e}{e}=1$ and further $\euler{e}{x}=0$ and $\euler{x}{e}=\euler{x}{[\OoX]}$ for all $x,y\in \Knullred(\XX)$.

The \define{Poincar\'{e} series} $p_R=\sum_{i=0}^\infty \dim_k(R_i)x^i$ of a fuchsian singularity $R$ of signature $(g;a_1,a_2...,a_t)$ was determined in \cite{Wagreich:1980}. As a rational function of $x$ it is given by the expression
$$
p_R=\frac{1+(g-2)x+(g-2)x^2+x^3}{(1-x)^2}+\frac{x^2}{(1-x)^2}\sum_{i=1}^t\frac{v_{a_i-1}}{v_{a_i}},
$$
where $v_a=\frac{x^n-1}{x-1}$, see \cite[Section 3]{Ebeling:2003}.

\begin{Proposition}\label{prop:Poinc}
Let $\XX=X\wt{a_1,a_2 ...,a_t}$ be a (WPC) of negative Euler characteristic. Let $\varphi_\XX=(1-x)^2\prod_{i=1}^tv_{a_i}$ the Coxeter polynomial of $\coh(\XX)$. Then the Coxeter polynomial $\varphi_\Tt$  of the fuchsian singularity $R$ attached to $\XX$ is given by
\begin{equation*}
\varphi_\Tt=p_R\cdot \varphi_\XX=\prod_{i=1}^tv_{a_i}\left[ 1+(g-2)x+(g-1)x^2+x^3+  \sum_{i=1}^t\frac{v_{a_i-1}}{v_{a_i}}   \right].
\end{equation*}
\end{Proposition}
For $t=0$ the expression for $\varphi_\Tt$ specializes to the Coxeter polynomial of $X$ from Example~\ref{ex:hypersurface}. For $g=0$ we obtain the Coxeter polynomial of the extended canonical algebra with the weight data $a_1,a_2,...,a_t$, see~\cite{Lenzing:Pena:2011}.

\begin{proof}
Comparing the Poincar\'{e} series $p_R$ with its K-theoretic variant $p_\XX=\sum_{i=0}^\infty \euler{\OoX}{\tau^i\OoX}x^i$, we first see that $p_R=(g+x)+p_\XX$. Next, applying a general result for one-point extensions \cite[Theorem 18.1]{Lenzing:1999}, we obtain $\varphi_\Tt=\varphi_\XX\cdot p_R$ as claimed.
\end{proof}

\begin{Proposition}\label{prop:fracCY}
If $R=k[x,y,z]/(f)$ is a fuchsian hypersurface, then $\tau^{\deg(f)}=[2]$ holds in $\Sing\ZZ{R}$, thus $\Sing\ZZ{R}$ is \define{fractionally Calabi-Yau} of CY-dimension $\frac{2}{\deg(f)}$. Accordingly, the Coxeter transformation of $R$ is periodic of period $\deg(f)$, and the Coxeter polynomial of $R$ is a product of cyclotomic polynomials.
\end{Proposition}
\begin{proof}
The first assertion is well known, compare \cite{Kussin:Lenzing:Meltzer:2013}; it immediately implies the claimed periodicity of the Coxeter transformation.
\end{proof}

\appendix
\section{The category of coherent sheaves on a smooth projective curve}\label{app:1}
\subsection{The category of coherent sheaves}
Let $k$ be an algebraically closed field, and $X$ a smooth projective curve, or for $k=\CC$ a compact Riemann surface. Then $X$ is equipped with a category $\Hh$ of  coherent sheaves (algebraic in case of a curve, holomorphic in case of a Riemann surface). We are giving here an axiomatic description of $\Hh$.
\begin{enumerate}
\item[$(H1)$] It is a connected $k$-linear category which is Hom- and Ext-finite over $k$.
\item[$(H2)$] $\Hh$ is hereditary (i.e. $\Ext^2(-,-)=0$) and noetherian (i.e. ascending chains of subobjects become stationary).
\item[$(H3)$] $\Hh$ satisfies Serre duality in the form $D\Ext^1(A,B)=\Hom(B,\tau A$, where $\tau$ is an auto-equivalence of $\Hh$.
\item[$(H4)$] The function field $K=k(X)$ of $\Hh$ (or $X$), see Section~\ref{sect:HNC}, has infinite dimension over $k$.
\end{enumerate}

\begin{Comment}
\begin{enumerate}[(i)]
 \item Due to $(H2)$ the category $\Hh$ has no non-zero projective or injective objects.
\item As a further consequence of Serre duality, the category $\Hh$ has almost split sequences; moreover the functor $\tau$ serves as the Auslander-Reiten translation.
 \item The sheaf $\Oo_X$ of regular functions on $X$ is called the \emph{structure sheaf} of $X$.
 \item The \emph{dualizing sheaf} $\om=\tau \Oo_X$ is the sheaf of differential forms of degree one over $X$. Accordingly $\tau^n\Oo_X=\om^{\tsr n}$, the n-fold tensor power of $\om$, is the sheaf of differential forms of degree $n$.
 \item The function field $\CC(X)$, classically defined as the global sections of the sheaf of meromorphic functions on $X$, allows the folowing alternative description: Let $\Hh_0$ be the Serre subcategory of all objects $Y$ of $\Hh$ having finite length (meaning that $Y$ has a finite filtration with factors that are simple in the category $\Hh$). Then the structure sheaf $\Oo_X$, when viewed as an object in the quotient category $\Hh/\Hh_0$, has an endomorphism ring, isomorphic to $K=k(X)$.
  \item Actually there is an equivalence $\Hh/\Hh_0\to\mmod(K)$ with $\Oo_X$ viewed as a member of of $\Hh/\Hh_0$ corresponding to $K=k(X)$. In particular $\Hh/\Hh_0$ is a \emph{length category}, that is, an abelian category in which each object has finite length.
 \item Condition $(H4)$ is introduced in order to avoid some degenerate examples like the category $\mod^{\ZZ\times\ZZ}(k[x,y])/\mmod_0^{\ZZ\times\ZZ}(k[x,y])$ which satisfies $(H1),(H2),(H3)$, but not $(H4)$.
\end{enumerate}
\end{Comment}
\begin{Definition}
The \emph{genus} $g_X$ of $\Hh$ (or $X$) is the $k$-dimension of $\Ext^1(\Oo_X,\Oo_X)$, equivalently the dimension of $\Hom(\Oo_X,\om)$.
\end{Definition}
Often the \emph{Euler characteristic} $\chi_X$ is more useful than the genus. It relates to the genus by means of $\chi_X=2(1-g_X)$.

\subsection{Euler form, rank, degree and Riemann-Roch}
In this section $\Hh=\coh(X)$ denotes the category of algebraic coherent sheaves on a smooth projective curve $X$ (resp.\ of holomorphic coherent sheaves on a compact Riemann surface). By $\Knull(\Hh)$ or just $\Knull(X)$ we denote the \emph{Grothendieck group} of the category $\Hh$ modulo short exact sequences. We always consider $\Knull(\Hh)$ to be equipped with the \emph{Euler form} which is the non-symmetric bilinear form given on classes of objects by the expression
$$
\euler{[A]}{[B]}=\dim_k\Hom(A,B)-\dim_k\Ext^1(A,B).
$$
This uses that $\Hh$ is hereditary. Often, we will write 
$\euler{A}{B}$ instead of $\euler{[A]}{[B]}$.

There are two important linear forms $\Knull(\Hh)\to\ZZ$, called rank and degree.

\begin{Definition}
The \emph{rank} of a coherent sheaf $A$, notation $\rk(A)$, is defined as the length of $A$, when viewed as an object of $\Hh/\Hh_0$.

The \emph{degree} of a coherent sheaf $A$ is defined as
$$
\deg(A)=\euler{\Oo_X}{A}-\rk(A)\cdot\euler{\Oo_X}{\Oo_X}.
$$
\end{Definition}
As follows from the next proposition, rank and degree satisfy largely dual properties.
\begin{Proposition}
The rank of a coherent sheaf $A$ is $\geq0$ with equality if and only if $A$ has finite length. Further,  we always have $\rk(A)=\rk(\tau A)$.

The degree of the structure sheaf is 0. The degree of any simple sheaf is 1.  A coherent sheaf of finite length always has degree $\geq0$ with equality if and only if $A=0$. Further, we have $\deg(A)=\deg(\tau A)$ if $A$ has finite length.

A coherent sheaf with rank and degree zero is already 0.~\hfill$\square$
\end{Proposition}

For each non-zero coherent sheaf $A$ its \define{slope} $\mu(A)=\frac{\deg(A)}{\rk(A}$ is thus well-defined.

In order to deal with Riemann-Roch, it is useful to introduce before the \em{reduced} (or \em{numerical}) Grothendieck group $\Knullred(X)$ which is the factor group of $\Knull(X)$ modulo the \emph{radical} $\{x\in \Knull(X)\mid \euler{x}{-}=0\}$ of the Euler form. Obviously, the Euler form descends from $\Knull(X)$ to $\Knullred(X)$, that is, denoting classes of coherent sheaves in $\Knullred(X)$ by double brackets, we have $\euler{\lsem X\rsem }{\lsem Y\rsem }=\euler{X}{Y}$ for all $X,Y \in \Hh$. Denoting the simple sheaf, concentrated in $x\in X$, by $\Ss_x$, it follows that all classes $\lsem \Ss_x\rsem $ of simple sheaves agree in $\Knullred(X)$. As a consequence $\Knullred(X)=\ZZ\lsem \Oo_X\rsem \oplus\ZZ\lsem \Ss\rsem $, where $\Ss$ is any simple sheaf. 

\begin{Theorem}[Riemann-Roch]
For any two coherent sheaves $A$, $B$ on $X$ we have
$$
\euler{\lsem A\rsem}{\lsem B\rsem}=(1-g_X)\rk(A)\rk(B)+
\left|\begin{array}{cc}\rk(A) &\rk(B) \\
\dg(A)&\dg(B)\\\end{array} \right|.
$$
\end{Theorem}
\proof It suffices to check the validity of the RR-formula for $A, B\in \{\Oo_X,\Ss\}$, where it is obviously satisfied.\endproof

\begin{Corollary}[Existence of morphisms]
Let $A$ and $B$ be non-zero vector bundles on $X$. Then 
$$
\frac{\euler{\lsem A\rsem}{\lsem B\rsem}}{\rk(A)\cdot\rk(B)}=(1-g_X)+(\mu(B)-\mu(A)).
$$
If $\mu(B)-\mu(A)>g_X-1=-1/2\chi_X$, then $\Hom(X,Y)\neq0$.~\hfill$\square$
\end{Corollary}

We recall that an object $E$ of $\Hh$ is called \emph{exceptional} if $\End(E)=k$ and $\Ext^1(E,E)=0$. In particular, an exceptional object is indecomposable.

\begin{Corollary}
The category $\coh(X)$ has an exceptional object $E$ if and only if $g_X=0$ and, additionally, $E$ has rank one, that is, $E$ is a line bundle. In particular, $\coh(X)$ has a tilting object if and only if $X$ is (isomorphic to) the Riemann sphere $\SSS$.
\end{Corollary}
\proof
Assume $E$ exceptional. By Riemann-Roch we obtain $1=\euler{E}{E}=(1-g_X)\rk(E)^2$, hence $g_X=0$ and $\rk(E)=1$. Concerning the last assertion, it is well known that $\coh(\SSS)$ has a tilting object $\Oo_X\oplus\Oo_X(x)$, where $\Oo_X(x)$ is the central term of 'the' non-split exact sequence $0\to \Oo_X\to \Oo_X(x)\to \Ss_x\to 0$.
\endproof

\section{The category of coherent sheaves on a weighted projective curve}\label{app:2}
Many aspects have already been dealt with in Sections~\ref{sect:WPC} and \ref{sect:HNC}. 

Let $\XX=(X,w)$ be an (WPC) and $\coh(\XX)$ its category of coherent sheaves. Let $\cohnull(\XX)$ denote the full subcategory of coherent sheaves of finite length. It is easily seen that the quotient categories $\coh(\XX)/\cohnull(\XX)$ and $\coh(X)/\cohnull(X)$ are equivalent. Hence the \define{function fields} $k(\XX)$ and $k(X)$ are isomorphic. In particular, the (WPC) $\XX$ determines the underlying smooth projective curve $X$.
The most important invariant of $\XX$, or $\coh(\XX)$, next to the function field $k(\XX)$, is the \define{orbifold Euler characteristic}, or just Euler characteristic $\chi_\XX$ of $\XX$.
The \define{Euler form} is the bilinear form on the Grothendieck group $\Knull(\XX)$ of the category $\coh(\XX)$ of coherent sheaves on $\XX$ which is given on (classes of) coherent sheaves by the expression $\euler{A}{B}=\dim \Hom(A,B)-\dim \Ext^{1}(A,B)$. 

For the convenience of the reader, we include the following treatment of orbifold Euler characteristic and weighted Riemann-Roch theorem from \cite{Lenzing:2018}.
Let $a_1,a_2,\ldots,a_t$ be the weight sequence associated to the weighted points of $\XX$. We put $\bar{a}=\lcm{a_1,a_2,\ldots,a_t}$. By means of the \define{averaged Euler form}
\begin{equation*}
\aveuler{X}{Y}=\frac{1}{\bar{a}}\sum_{j=0}^{\bar{a}-1}\euler{\tau^j X}{Y},
\end{equation*}
we define the \define{degree} as the linear form deg on $\Knull(\XX)$, given by the expression\footnote{The factor $\bar{a}$ is introduced to yield integer values for the degree.}
$$
\dg(A)=\bar{a}\aveuler{\OoX}{A}-\bar{a}\aveuler{\OoX}{\OoX}.
$$
Rank and degree then are uniquely determined by the following complementry properties:
\begin{enumerate}[(i)]
\item For $A$ in $\coh(\XX)$ we always have $\rk{A}\geq0$, further the structure sheaf $\OoX$ has rank one. Moreover, $\rk{A}=0$ holds exactly if $A$ has finite length. Also $\rk{}$ is preserved under automorphisms of $\coh(\XX)$, in particular under the Auslander-Reiten translation.
\item A non-zero object $A$ of finite length has degree $>0$, and the structure sheaf has degree zero. Further each simple sheaf $S_x$, concentrated in a point $x$ of $\XX$, has degree $\frac{\bar{a}}{w(x)}$.
\end{enumerate} 
We next define the \define{orbifold Euler characteristic} of $\XX$ as $\chi_\XX=\frac{2}{\bar{a}}\aveuler{\OoX}{\OoX}$. Clearly, for non-weighted smooth projective curves the orbifold Euler characteristic agrees with the ordinary Euler characteristic.

\begin{Theorem}\label{thm:OrbifoldEulerChar}
Let $\XX=X\wt{a_1,a_2,\ldots,a_t}$ be a weighted projective curve with underlying smooth projective curve $X$. Then the orbifold Euler characteristic of $\XX$ is given by the expression
\begin{equation}\label{eqn:OrbifoldEulerChar}
  \chi_\XX=\frac{2}{\bar{a}}\aveuler{\OoX}{\OoX}=\chi_X-\sum_{i=1}^{t}\left(1-\frac{1}{a_i}\right).
\end{equation}
\end{Theorem}

\begin{proof}
We work in the reduced Grothendieck group $\Knullred(\XX)$. We put $a=\rcl{\OoX}$, $s_0=\rcl{S_x}$ where $S_x$ is the simple sheaf concentrated in a non-weighted point, $s_i=\rcl{S_i}$, where $S_i$ is the simple sheaf concentrated in $x_i$ with $\Hom(\OoX,S_i)=k$ ($i=1,2,\ldots,t$). Then we have
\begin{equation}\label{eqn:Tau}
  \tau a -a = -\sum_{i=1}^{t}s_i+(t-\chi_X)s_0,
\end{equation}
where $\chi_X=2\euler{a}{a}=2(1-g)$ is the Euler characteristic of the underlying smooth projective curve $X$. The validity of formula \eqref{eqn:Tau} is easily checked by forming for both sides the Euler product with each member of the generating system $a,s_0,\tau^j s_i$, $i=1,2,\ldots,t$ and using non-degeneracy of the Euler form on $\Knullred(\XX)$. Then, applying $\sum_{j=0}^{\bar{a}-1}\tau^j$ to \eqref{eqn:Tau} implies
\begin{equation}\label{eqn:TauUpP}
  \tau^{\bar{a}}a-a=\bar{a}\,\delta\, s_0,\text{ where }\delta=\sum_{j=1}^{t}(1-\frac{1}{a_i})-\chi_X.
\end{equation}
Finally, we obtain
\begin{equation}
  2\aveuler{a}{a}+\bar{a}\delta = \aveuler{a}{a}+\aveuler{a}{\tau^{\bar{a}}a} =
  \frac{1}{\bar{a}}\left(\sum_{j=0}^{\bar{a}-1}\euler{\tau^ja}{a}-\sum_{j=0}^{\bar{a}-1}\euler{\tau^{(\bar{a}-1)-j}a}{a}\right)  =0,
\end{equation}
proving the claim.
\end{proof}
We note that the rank allows a purely K-theoretic definition as $\rk(A)=\euler{A}{S}$ where $S$ is a simple sheaf concentrated in an ordinary point. This K-theoretic rank agrees with the sheaf-theoretic rank, defined for a coherent sheaf $A$ as its length in the Serre quotient $\coh(\XX)/\cohnull(\XX)$.
\begin{Theorem}[Riemann-Roch]\label{thm:WeightedRR}
Assume $b,c$ are members of $\Knull(\XX)$, where $\XX=X\wt{a_1,a_2,\ldots,a_t}$ then
\begin{equation}\label{eqn:WeightedRR}
  \aveuler{b}{c}=\aveuler{\OoX}{\OoX}\cdot\rk{b}\cdot \rk{c}+\frac{1}{\bar{a}}\cdot\begin{vmatrix}
\rk(b) \ & \rk(c) \\
\dg(b)& \dg(c)
\end{vmatrix}.
\end{equation}
\end{Theorem}
\begin{proof}
Passing to the reduced Grothendieck group $\Knullred(\XX)$ one checks the equality of both sides of \eqref{eqn:WeightedRR} on the generating system $a,s_0,\tau^js_i$, $i=1,2,\ldots,t$, $j\in\ZZ_{a_i}$.
\end{proof}


\providecommand{\MRhref}[2]{%
  \href{http://www.ams.org/mathscinet-getitem?mr=#1}{#2}
}
\providecommand{\href}[2]{#2}

\end{document}